\documentclass[12pt]{amsart}
\usepackage{amsmath,amsthm}
\usepackage{amsfonts}
\usepackage[psamsfonts]{amssymb}

\title[Abelian transvection groups]{Invariant theory of abelian transvection groups}
\author{Abraham Broer}
\address{D\'epartement de math\'ematiques\\ et de statistique\\
Universit\'e de Montr\'eal\\
C.P. 6128, succursale Centre-ville\\
Montr\'eal (Qu\'ebec)\\
Canada H3C 3J7}

\email{broera@DMS.UMontreal.CA}
\date\today
\thanks{The author wishes to thank Jianjun Chuai for some interesting discussions on the
Hilbert ideal}

\theoremstyle{plain}
\newtheorem{lemma}{Lemma}
\newtheorem{proposition}{Proposition}
\newtheorem*{theorem}{Theorem}

\newtheorem*{corollary}{Corollary}

\theoremstyle{remark}

\newtheorem{example}{Example} 

\def\F{{\mathbb F}}                            %Galois field
\def\Z{{\mathbb Z}}                            %integers
\def\dim{\mathop{\operatorname{dim}}\nolimits} %dimension

\def\GL{\mathop{\operatorname{GL}}\nolimits}

\def\Hilb{{\mathfrak H}} %Hilbert ideal

\def\Tr{\mathop{\operatorname{Tr}}\nolimits}   %transfer or trace
\def\Tau{{\mathcal T}}

\begin{document} 
\maketitle
\begin{abstract}
Let $G$ be a finite group acting linearly on the vector space $V$ over a field of
arbitrary characteristic. The action is called {\em coregular} if the 
invariant ring is generated by algebraically independent homogeneous invariants and
 the {\em direct summand property} holds if 
there is a  surjective $k[V]^G$-linear map $\pi:k[V]\to k[V]^G$.

The following Chevalley--Shephard--Todd type theorem is proved.
Suppose $G$ is abelian, then the action is coregular if and only if  $G$ is generated by
pseudo-reflections and the direct summand property holds.
\end{abstract}

\section*{Introduction}
Let  $V$ be a  vector space of dimension $n$ over a field $k$. A linear transformation $\tau:V\to V$ is called a {\em pseudo-reflection}, if its fixed-points space $V^\tau=\{v\in V; \tau(v)=v\}$ is a linear subspace
of codimension one. Let $G<\GL(V)$ be a finite group acting linearly on $V$. Then $G$ acts by algebra automorphisms on the coordinate ring $k[V]$, which is by definition
the symmetric algebra on the dual vector space $V^*$. 
We shall say that $G$ is a {\em pseudo-reflection group}
if $G$ is generated by pseudo-reflections; it is called a {\em non-modular } group if $G$ is not divisible
by the characteristic of the field. The action is called {\em coregular} if the 
invariant ring is generated by $n$ algebraically independent homogeneous invariants.
Finally we say that the {\em direct summand property} holds if 
there is a  surjective $k[V]^G$-linear map $\pi:k[V]\to k[V]^G$ respecting the gradings.

For a non-modular group the direct summand property always holds, because in that case
we can take the {\em transfer} $\Tr^G$ as projection, defined by
$$\Tr^G: k[V]\to k[V]^G:\ \Tr^G(f)=\sum_{\sigma\in G}\sigma(f);$$
since for any invariant $f$ we have $\Tr^G({|G|^{-1}f)}=f$.
A theorem of Serre~\cite[Theorem 6.2.2]{Benson} implies that if the action is coregular then $G$ is a pseudo-reflection group and the
direct summand property holds. We conjectured that the converse also holds, cf.~\cite{Broer2005}. The theorem
of Chevalley-Shephard-Todd~\cite[Chapter 6]{Benson} says that the converse holds if the group is non-modular.
In this note we prove that the converse holds if $G$ is abelian. Elsewhere we show that the
converse is also  true if $V$ is an irreducible $kG$-module, cf.~\cite{BroerS}.

\begin{theorem}\label{maintheorem}
Suppose $G<\GL(V)$ is an {\em abelian} group acting on the finite dimensional vector space $V$. Then the action is coregular if and only if $G$ is a pseudo-reflection group and the
direct summand property holds.
\end{theorem}

As corollary we get a special case of a conjecture made by Shank--Wehlau, cf.~\cite{ShankWehlau1997}. 
Suppose the characteristic of the field is $p>0$.

\begin{corollary}
Let $G<\GL(V)$ be an abelian $p$-group acting linearly on the vector space $V$.
 The image of the transfer map $\Tr^G$ is a principal ideal in $k[V]^G$ if and only if the action is coregular.
\end{corollary}

\section{Hilbert ideal and the direct summand property.}
For elementary facts on the invariant theory of finite groups we refer to~\cite{Benson},
for a discussion of the direct summand property and the different see~\cite{Broer2005}. We recall that the different $\theta_G$ of the action can be defined as the largest degree homogeneous
form in $k[V]$ such that $\Tr^G(\frac{f}{\theta})\in k[V]^G$ for all $f\in k[V]^G$; it is unique up to a multiplicative scalar. The direct summand property holds if and only if there exists a $\tilde{\theta}_G$ such
that $\Tr^G(\frac{\tilde{\theta}_G}{\theta_G})=1$ and then we can take as $k[V]^G$-linear projection
$$\pi:k[V]\to k[V]^G:\ \pi(f):=\Tr^G(\frac{\tilde{\theta}_Gf}{\theta_G}).$$

 If $J\subseteq k[V]^G$ is an ideal, we
define $J^e:=I\cdot k[V]$, the ideal in $k[V]$ generated by $J$. If $I\subseteq k[V]$, we define $I^c:=I\cap k[V]^G$, the ideal in $k[V]^G$ generated by the invariants contained in $I$. An important
consequence of the direct summand property is that it implies $J=J^{ec}$, 
cf.~\cite[Proposition 6]{Broer2005}.

The {\em Hilbert ideal} $\Hilb\subset k[V]$ is the ideal generated by all positive degree homogeneous invariants. Hilbert already noticed that if the direct summand property holds then any
collection of homogeneous $G$-invariants generating the Hilbert ideal also generates
the algebra of invariants.
We say that the Hilbert ideal is a {\em complete intersection ideal}, if it can be generated by $n$ homogeneous invariants, where $n=\dim V$. Those invariants necessarily form a (very special) homogeneous system of parameters. We shall use the following criterion for coregularity.

\begin{proposition}\label{hilbert ideal}
The action is coregular  if and only if  the Hilbert ideal
$\Hilb$ is a complete intersection ideal and the direct summand property holds.
\end{proposition}

\begin{proof}
If the action is coregular, then $k[V]^G=k[f_1,\ldots,f_n]$ and so $\Hilb=(f_1,\ldots,f_n)$
is a complete intersection ideal. Coregularity  also implies the direct summand property,
cf.\cite[Proposition 5(ii)]{Broer2005}.

Conversely, suppose the direct summand property holds and $\Hilb=(f_1,\ldots,f_n)$, where
$f_1,\ldots,f_n$ are homogeneous invariants of positive degree. Now we recall Hilbert's argument showing that $R:=k[f_1,\ldots.f_n]$ is equal to $k[V]^G$. Suppose $R$ is not equal to $k[V]^G$.
Let then $f\in k[V]^G$ be of minimal degree such that $f$ is not in $R$. But $f\in \Hilb$, so there
are $h_1,\dots,h_n\in k[V]$ of degree strictly smaller than the degree of $f$,
 such that $f=h_1f_1+\ldots+h_nf_n$. By hypothesis there
is a $k[V]^G$-linear projection operator $\pi:k[V]\to k[V]^G$ respecting grading.  We can assume $\pi(1)=1$.
We use it to get
$$f=\pi(f)=\pi(h_1)f_1+\ldots+\pi(h_n)f_n.$$
Each $\pi(h_i)$ is now invariant and of strictly lower degree than $f$, hence is in $R$. But then
$f\in R$, which is a contradiction. It follows that $k[V]^G$ is generated by $f_1,\ldots,f_n$, and
so the action is coregular.
\end{proof}

Let $U\subseteq V^G$ be a linear subspace, and $U^\perp\subset V^*=k[V]_1$ the space
of linear forms vanishing on $U$. Let $I(U)$ be the ideal in $k[V]^G$ generated by $U^\perp$.
We shall define $\Hilb_U$, the {\em Hilbert ideal relative to $U$}, to be
$I(U)^{ce}$, i.e., $\Hilb_U$ is the ideal of $k[V]$ generated
by all the invariants contained in $I(U)$.
In particular, for $U=\{0\}$ we
get the original Hilbert ideal $\Hilb$.
Let $s$ be the codimension of $U$ in $V$, then we say that $\Hilb_U$ is a {\em complete intersection ideal} if it can be generated by $s$ homogeneous invariants.

\begin{lemma}\label{U hilbert ideal}
Let $\Hilb_U$ be the Hilbert ideal relative to $U\subset V^G$.  If $\Hilb_U$ is a complete intersection ideal then the Hilbert ideal  $\Hilb$ is also a complete intersection ideal.
\end{lemma}

\begin{proof}
We shall use that the quotient algebra $k[V]^G/I(U)^c$ is a polynomial ring,
a result due to Nakajima~\cite[Proof of Lemma 2.11]{Nakajima1983}. We recall the quick proof.

To prove this result we can suppose that $k$ is algebraically closed so that we can use the language of algebraic geometry.
Let $\pi_G:V\to V/G$ be the quotient map. The linear algebraic group $U$ acts on $V$ by translations:
$$U\times V\to V: (u,v)\mapsto  u+v.$$ Since $U\subseteq V^G$, the translations commute with the $G$-action on
$V$, hence the $U$-action on $V$ descends to an action on the quotient variety
$$U\times V/G\to V/G: (u,\pi_G(v))\mapsto \pi_G(u+v).$$
It acts simply transitive on itself and on its image $\pi_G(U)$ in $V/G$. So $\pi_G(U)$ is isomorphic
to $U\simeq k^{n-s}$, hence the coordinate ring of $\pi_G(U)$ is isomorphic to a polynomial
ring with $n-s$ variables. The coordinate ring of $V/G$ can be identified
with $k[V]^G$ and then $\pi_G(U)$ is defined by $I(U)^c$.  It follows that
$k[V]^G/I(U)^c$ is a polynomial ring in $n-s$ variables. This finishes the proof of Nakajima's result.

So we can find $n-s$ homogeneous invariants $f_{s+1}, f_{s+2},\ldots, f_n$ such that 
$$I(U)^c+(f_{s+1}, f_{s+2},\ldots, f_n)k[V]^G=k[V]^G_+,$$
the maximal homogeneous ideal of $k[V]^G$. So 
$$\Hilb=(k[V]^G_+)^e=I(U)^{ce}+(f_{s+1}, f_{s+2},\ldots, f_n)k[V]=\Hilb_U+(f_{s+1}, f_{s+2},\ldots, f_n)k[V].$$
Now if  $\Hilb_U$ is a complete intersection ideal, hence generated by $s$ elements,
it follows that $\Hilb$ is generated by $n$ elements and is also a complete intersection ideal.
\end{proof}

\section{Abelian transvection groups}
For any pseudo-reflection $\rho$ on $V$ there is a vector $e_\rho\in V$ such that $(\rho -1)(V)=k e_\rho$
and a functional $x_{\rho}\in V^*$ such that
$\rho(v)-v=x_{\rho}(v)e_{\rho}.$
Then $v\in V^\rho$ if and only if $x_{\rho}(v)=0$, or $x_{\rho}$ is a linear form defining the
fixed-points set $V^\rho$.
There also is
a unique linear map $\Delta_\rho:k[V]\to k[V]$ such that for $f\in k[V]$:
$$\rho(f)-f=\Delta_{\rho}(f) x_{\rho}.$$

The pseudo-reflection is called a {\em transvection} if $\rho(e_\rho)=e_{\rho}$, i.e., $e_\rho\in V^{\rho}$,
or equivalently if $\Delta_\rho(x_\rho)=0$. The fixed-points set $V^\rho$ is then called a transvection
hyperplane.
Otherwise the pseudo-reflection is diagonalisable over $k$, and called {\em homology}, i.e., there is a basis of $V$ consisting of eigenvectors.
A {\em transvection group} is a group generated by transvections.

\begin{proposition}\label{transvectiongroups}
Let $G$ be a finite abelian transvection group acting on $V$.

(i) $\Hilb_{V^G}$ is a complete intersection ideal, where $\Hilb_{V^G}$ is the Hilbert ideal relative to $V^G$.

(ii) $G$ is an abelian  $p$-group, where $p$ is the characteristic of the field.
\end{proposition}

\begin{proof}
(i) Let $r_1$ and $r_2$ be two transvections in $G$, whose fixed-point sets are defined by the two linear forms $x_1$ and $x_2$. Then for any $f\in k[V]$ there is a unique
$\Delta_1(f)$ and $\Delta_2(f)$ such that $r_i(f)=f+\Delta_i(f)x_i$, for $i=1,2$.
Since the $r_i$ are transvections we have $\Delta_i(x_i)=0$.
For any linear form $y$ we have that $\Delta_i(y)$ is a scalar and
\begin{eqnarray*}
r_1(r_2(y))&=&r_1\left(y+\Delta_2(y)x_2\right)\\
&=&y+\Delta_1(y)x_1+\Delta_2(y)x_2+\Delta_2(y)\Delta_1(x_2)x_1\\
r_2(r_1(y))&=&r_2\left(y+\Delta_1(y)x_1\right)\\
&=&y+\Delta_2(y)x_2+\Delta_1(y)x_1+\Delta_1(y)\Delta_2(x_1)x_2.
\end{eqnarray*}
Since $G$ is abelian we get for all $y\in V^*$ that
$$\Delta_2(y)\Delta_1(x_2)x_1=\Delta_1(y)\Delta_2(x_1)x_2.$$

If $x_1$ and $x_2$ are dependent then $\Delta_i(x_j)=0$. Supposing they are independent we get
$\Delta_2(y)\Delta_1(x_2)=0$ for all linear forms $y$, hence $\Delta_1(x_2)=0$. Similarly $\Delta_2(x_1)=0$.
Therefore we get $r_i(x_j)=x_j$. Since our group is an abelian transvection group, it follows that any linear form defining a transvection hyperplane is a $G$-invariant linear form.

Let $\Tau\subset G$ be the collection of transvections in $G$.
For any  $\tau\in \Tau$ fix $x_\tau$ as above. Since the transvections generate $G$ we get
$$(V^G)^{\perp}=(\cap_{\tau \in \Tau} V^\tau)^\perp=
\sum_{\tau \in \Tau} (V^\tau)^\perp=\sum_{\tau \in \Tau}<x_\tau>=<x_\tau;\ \tau\in\Tau>.$$
Since we just proved that each $x_{\tau}\in (V^*)^G\subseteq k[V]^G$ it follows
that $(V^G)^{\perp}$ is generated by linear invariants, say $x_1,\ldots,x_{n-s}$ and so
$\Hilb_{V^G}$ is a complete intersection ideal, since
$$I(V^G)=(x_1,\ldots,x_{n-s})=I(V^G)^{ce}=\Hilb_{V^G}.$$

(ii)
Suppose $G$ is not a $p$-group, then (by extending the field if necessary) there exists a $\sigma\in G$ and a linear form $y\in V^*$ such that $\sigma(y)=cy$, where $c\neq 1$. Since $G$ is generated
by transvections, there must be a transvection $\tau\in G$, with corresponding $x_\tau$ and 
$\Delta_\tau$, such that $\tau(y)\neq y$, or $\Delta_\tau(y)\neq 0$. Then
\begin{eqnarray*}
\sigma\tau(y)&=&\sigma(y+\Delta_\tau(y) x_\tau)=cy+\Delta_\tau(y)\sigma(x_\tau)\\
\tau\sigma(y)&=&\tau(cy)=cy+\Delta_\tau(y) cx_\tau.
\end{eqnarray*}
Comparing we get $\sigma(x_\tau)=cx_\tau$ and so $x_\tau\not\in (V^*)^G$, which contradicts (i).
So $G$ is a $p$-group.
\end{proof}

\section{Reduction to abelian transvection groups and diagonalisable pseudo-reflection groups}
The following lemma allows us to treat separately abelian transvection groups and diagonalisable pseudo-reflection groups. The first two parts were known to Nakajima~\cite[Proof of Proposition 2.1]{Nakajima1982}.

\begin{proposition}\label{reduction}
Let $G<\GL(V)$ be an abelian pseudo-reflection group $G$ acting on $V$.
Denote $T$ for  the subgroup of $G$ generated by the
transvections and $D$ for the subgroup generated by the homologies in $G$. 

(i) Then $D$ is a non-modular group, $T$ is a $p$-group and $G=T\times D$.

(ii) There is a direct sum decomposition of $kG$-modules $V=V^D\oplus V_D$,
where $D$ acts trivially on $V^D$ and $T$ acts trivially on $V_D$. For the invariant rings
we get
$$k[V]^G\simeq k[V^D]^T\otimes k[V_D]^D.$$
Consequently, the $G$-action on $V$ is coregular if and only if the $T$-action on $V^D$ (or on $V$)
and the $D$-action on $V_D$ (or on $V$) are coregular.

(iii) The direct summand property holds for the $G$-action on $V$ if and only if
the direct summand property holds for the $T$-action on $V^D$ (or $V$).
\end{proposition}

\begin{proof} 
(i) Since every generator of $D$ is diagonalisable over $k$ and $D$ is abelian, the group $D$ is simultaneously diagonalisable; in particular it is non-modular. Since $T$ is
an abelian transvection group, it is a $p$-group by Lemma~\ref{transvectiongroups}. So $T\cap D=\{1\}$
and $G=T\times D$.

(ii)
Let $V^D$ be the space of invariants and $V_D$ the  direct sum of the remaining eigenspaces
of $D$, so at least $V=V^D\oplus V_D$ as $kG$-modules.

If $\tau\in T$ then by commutativity also $\tau(v)\in V^D$,
so $V^D$ is a $kG$-submodule.

Let $\tau$ be transvection  with corresponding $e_\tau\in V$ and $x_\tau\in V^*$ such that
 $\tau(v)-v=\delta(v)e_\tau$, for any $v\in V$.
Let $\sigma$ be a homology and $\sigma v=c v$, where $v$ is the eigenvector for $\sigma$ with eigenvalue $c\neq 1$.
Then
$\tau\sigma v=\tau cv=cv +x_\tau(v) ce_\tau$ and
$\sigma \tau v=\sigma(v+x_\tau(v)e_\tau)=cv+x_\tau(v)\sigma(e_\tau)$.
Commutativity implies $x_\tau(v)(\sigma(e_\tau)-c e_\tau)=0$. If $x_\tau(v)\neq 0$
it follows that $e_\tau$ is an eigenvector for $\sigma$ with eigenvalue $c$. So $v$ is a scalar multiple
of $e_\tau$ (since $\sigma$ is a homology, the eigenspace with eigenvalue $c\neq 1$ is one-dimensional). But since $e_\tau\in V^\tau$ (since $\tau$ is a transvection) it follows that $\tau(v)=v$ and so $x_v(v)=0$. Contradiction.
So necessarily $x_\tau(v)=0$ and $\tau(v)=v$. 

Since the eigenvectors of homologies
with non-identity eigenvalue span $V_D$ (since those homologies generate $D$) it follows that $T$ acts trivially on $V_D$. In particular
$V_D$ is also a $kG$-submodule and $V=V^D\oplus V_D$ is a decomposition as $kG$-modules.

Let $y_1,\ldots,y_m$ be a basis of linear forms vanishing on $V_D$,
and $z_1,\ldots,z_{n-m}$ a basis of linear forms vanishing on $V^D$. So
$y_1,\ldots,y_m$ are coordinate functions on $V^D$, 
 $z_1,\ldots,z_{n-m}$  are coordinate functions on $V_D$ and
$$k[V]=k[y_1\ldots,y_m,z_1,\ldots,z_{n-m}]=k[y_1,\ldots,y_n]\otimes k[z_1,\ldots,z_{n-m}]=k[V_D]\otimes k[V^D].$$
For the invariants we get
$$k[V]^G\simeq k[V^D]^T\otimes k[V_D]^D.$$

(iii) The different of the $G$-action $\theta_G$ is a product of linear forms $x_\alpha$, where the zero-set of
$x_\alpha$, say $V_\alpha:=\{v\in V; x_\alpha(v)=0\}$, is the fixed-point set of a pseudo-reflection, cf.~\cite[Proposition 9]{Broer2005}.
The same for $\theta_T$ and $\theta_D$.
If $\tau$ is a transvection, then $V^\tau\supset V_D$; if $\tau$ is a diagonalisable then
$V^\tau\supset V^D$. It follows that 
$\theta_T\in k[y_1,\ldots,y_m]=k[V^D]$ and
$\theta_D\in k[z_1,\ldots,z_{n-m}]\in k[V_D]$ and
$\theta_G=\theta_T\cdot \theta_D$. In particular $T$ acts trivially on $\theta_D$ and 
$D$ acts trivially on $\theta_T$.

Suppose the direct summand property holds for the $G$-action, i.e, there exists
a $\tilde{\theta}_G\in k[V]$ such that $\Tr^G(\frac{\tilde{\theta}_G}{\theta_G})=1$.
Put $\hat{\theta}_T:=\Tr^D(\frac{\tilde{\theta}_G}{\theta_D})$, then
$$\Tr^T\left(\frac{\hat{\theta}_T}{\theta_T}\right)=
\Tr^T\left({\frac{1}{\theta_T}} \Tr^D(\frac{\tilde{\theta}_G}{\theta_D})\right)=
\Tr^T\left(\Tr^D(\frac{\tilde{\theta}_G}{\theta_D})\right)=
1,$$
since $\Tr^G=\Tr^T\circ\Tr^D$ and $\theta_T$ is $D$-invariant.
So the direct summand property holds for the $G$-action $V$.

Suppose that $\hat{\theta}_T$ is not in $k[V^D]=k[y_1,\ldots,y_n]$, so we can write 
$$\hat{\theta}_T=\tilde{\theta}_T+ \sum_{i=1}^{n-m} z_i f_i,$$
where $\tilde{\theta}_T\in k[V^D]$ and $f_i\in k[V]$. 
Then 
$$1=\Tr^T\left(\frac{\hat{\theta}_T}{\theta_T}\right)=\Tr^T\left(\frac{\tilde{\theta}_T}{\theta_T}\right)+
 \sum_{i=1}^{n-m} z_i\Tr^T\left(\frac{f_i}{\theta_T}\right)=\Tr^T\left(\frac{\tilde{\theta}_T}{\theta_T}\right),$$
 since $\Tr^T\left(\frac{f_i}{\theta_T}\right)$ is of negative degree, hence $0$. It follows
 that the direct summand property also holds for the $T$-action on $V^D$.
 
 Conversely, suppose the direct summand property holds for the $T$-action on $V$ then
 by the foregoing argument  the direct summand property also holds for the $T$-action on $V^D$).
 Hence there is a $\tilde{\theta}_T\in k[y_1,\dots,y_m]$ such that 
 $\Tr^T\left(\frac{\tilde{\theta}_T}{\theta_T}\right)=1$. 
 Put $$\tilde{\theta}_G:=|D|^{-1}\cdot\theta_D\cdot \tilde{\theta}_T,$$ this makes sense since $D$ is non-modular.
 Then
  $$\Tr^G\left(\frac{\tilde{\theta}_G}{\theta_G}\right)=
  \Tr^T\circ \Tr^D\left(\frac{|D|^{-1}\cdot\theta_D\cdot \tilde{\theta}_T}{\theta_T\cdot\theta_D}\right)
  = \Tr^T\left(\frac{\tilde{\theta}_T}{\theta_T}   \Tr^D(|D|^{-1})\right)=1$$
  and so the direct summand property also holds for the $G$-action on $V$.
\end{proof}

\section{Proofs of main results}
We now prove our main theorem and its corollary.

\begin{theorem}
Suppose $G<\GL(V)$ is an {\em abelian} group acting on the finite dimensional vector space $V$. Then the action is coregular if and only if $G$ is a pseudo-reflection group and the
direct summand property holds.
\end{theorem}

\begin{proof}
Even when $G$ is not abelian by Serre it is generally true that if the action is coregular
then $G$ acts as a pseudo-reflection group and the direct summand property holds, cf.~\cite{Broer2005}.

Suppose that $G$ is an abelian pseudo-reflection group and the direct summand property holds.
Then $G=T\times D$, where $T$ is the subgroup generated by transvections and $D$ the
subgroup generated by diagonalisable reflections, as in Proposition~\ref{reduction}. We use
the notation of that lemma . Since $D$ is a non-modular pseudo-reflection group acting on $V_D$, it follows
from the classical Chevalley-Shephard-Todd theorem that $k[V_D]^D$ is a polynomial ring.
From Proposition~\ref{reduction} it also follows that $T$ is an abelian transvection group
acting on $V^D$ and that this action has the direct summand property. From
Proposition~\ref{transvectiongroups} and Lemma~\ref{U hilbert ideal} it follows that the Hilbert ideal $\Hilb$ of this action is a complete intersection ideal. So by the criterion in
Proposition~\ref{hilbert ideal} it follows that the $T$-action on $V^D$ is coregular, and
so $k[V^D]^T$ is a polynomial ring. So $k[V]^G=k[V^D]^T\otimes k[V_D]^D$ (see 
Proposition~\ref{reduction} again) is a polynomial ring. Hence the $G$-action
is coregular.
\end{proof}

We get a special case of Shank-Wehlau's conjecture~\cite{ShankWehlau1997}.

\begin{corollary}
Let $G<\GL(V)$ be an abelian $p$-group acting linearly on the vector space $V$.
 The image of the transfer map $\Tr^G$ is a principal ideal in $k[V]^G$ if and only if the action is coregular.
\end{corollary}

\begin{proof}
In~\cite{Broer2005} it was already shown for $p$-groups that the direct summand property
holds if and only if the image of the transfer map $\Tr^G$ is a principal ideal in $k[V]^G$ and
that this condition implies that $G$ is a transvection group, and if $G$ is abelian the
theorem above implies that the action is even coregular. 
Conversely, if the action is coregular, then the direct summand property holds and
the image of the transfer is a principal ideal.
\end{proof}

\begin{example}
The simplest example of an abelian  transvection group that satisfies neither the direct summand
property nor the coregularity property is the following.
Take $p=2$, $k=\F_2$, $G=<\sigma_1,\sigma_2,\sigma_3>\simeq (\Z/2\Z)^3$, $V=\F_2^3$ and the action is defined by the three matrices
$$\sigma_1=
\left(\begin{matrix}1&0&0&0\\ 0&1&0&0 \\ 1&0&1&0\\ 0&0&0&1\end{matrix}\right);\ 
\sigma_2=
\left(\begin{matrix}1&0&0&0\\ 0&1&0&0 \\ 0&0&1&0\\ 0&1&0&1\end{matrix}\right);\ 
\sigma_3=
\left(\begin{matrix}1&0&0&0\\ 0&1&0&0 \\ 1&1&1&0\\ 1&1&0&1\end{matrix}\right).
$$
In fact $\sigma_1$, $\sigma_2$ and $\sigma_3$ are the only transvections in the group, with transvection hyperplanes defined by $x_1$, $x_2$ and $x_1+x_2$ respectively. So the ideal $I$ defining $V^G$ is $I=(x_1,x_2)$ and the Dedekind different is $\theta=x_1x_2(x_1+x_2)$.
A minimal generating  set of invariants is, cf.~\cite{Nakajima1979}, $x_1$, $x_2$ and
\begin{eqnarray*}
f_3&:=&x_1x_3(x_1+x_3)+x_2x_4(x_2+x_4);\\
N(x_3)&=&x_3(x_3+x_1)(x_3+x_2)(x_3+x_1+x_2);\\
N(x_4)&=&x_4(x_4+x_1)(x_4+x_2)(x_4+x_1+x_2).
\end{eqnarray*}
There is one generating relation among the generators.

The Hilbert ideals  are complete intersection ideal
$\Hilb=(x_1,x_2,N(x_3),N(x_4))$, and $\Hilb_{V^G}=(x_1,x_2)$. 
But the direct summand property does not hold, since if it would hold we would have
for $J=(x_1,x_2)k[V]^G$ that $J=J^{ec}$, but $J^{ec}=(x_1,x_2,f_3)k[V]^G$. 
Or more directly, a calculation shows that if $f\in k[V]$ is of degree $3$ then
$\Tr^G(\frac{f}{\theta_G})=0$.
\end{example}

\end{document}